\documentclass[11pt]{article}
\usepackage{amsmath, amscd, amssymb, mathrsfs, latexsym, epsfig, color, amsthm, mathtools, tikz-cd, array, adjustbox, bbm, tikz, enumitem, relsize, xcolor}

\usepackage{titling}


\usepackage[normalem]{ulem}

\makeatletter
\def\moverlay{\mathpalette\mov@rlay}
\def\mov@rlay#1#2{\leavevmode\vtop{%
		\baselineskip\z@skip \lineskiplimit-\maxdimen
		\ialign{\hfil$\m@th#1##$\hfil\cr#2\crcr}}}
\newcommand{\charfusion}[3][\mathord]{
	#1{\ifx#1\mathop\vphantom{#2}\fi
		\mathpalette\mov@rlay{#2\cr#3}
	}
	\ifx#1\mathop\expandafter\displaylimits\fi}
\makeatother

\newcommand{\bigcupdot}{\charfusion[\mathop]{\bigcup}{\cdot}}

\numberwithin{equation}{section}
\newtheorem{theorem}{Theorem}[section]
\newtheorem{proposition}[theorem]{Proposition}

\newtheorem{corollary}[theorem]{Corollary}

\newtheorem{lemma}[theorem]{Lemma}

\newtheorem{claim}[theorem]{Claim}

\theoremstyle{definition}

\newtheorem{remark}[theorem]{Remark}

\newcommand{\R}{{\mathbb R}}

\tikzset{
	labl/.style={anchor=south, rotate=90, inner sep=.5mm}
}

\title{A Proof of Gr\"unbaum's Lower Bound Conjecture for general polytopes}
\author{Lei Xue\thanks{This research was partially supported by a graduate fellowship from NSF grant DMS-1664865.}\\
	\small \texttt{lxue@uw.edu}
}

\begin{document}
	%%%%%%%%%%%%%%%%%%%%%%%%%%%%%%%%%%%%%%%%%%%%
	%%%%%%%%%%%%%%%%%%%%%%%%%%%%%%%%%%%%%%%%%%%%
	\maketitle
	
	\begin{abstract}
In 1967, Gr\"unbaum conjectured that any $d$-dimensional polytope with $d+s\leq 2d$ vertices has at least
\[\phi_k(d+s,d) = {d+1 \choose k+1 }+{d \choose k+1 }-{d+1-s \choose k+1 } \] $k$-faces. We prove this conjecture and also characterize the cases in which equality holds.
	\end{abstract}

\section{Introduction}\label{sect:intro}
The paper is devoted to the proof of Gr\"unbaum's general lower bound conjecture for polytopes with few vertices.​

​
In the last fifty years a lot of effort has gone into trying to understand face numbers of polytopes. For instance, McMullen \cite{MR283691} established the Upper Bound Theorem in 1970, which provides tight upper bounds on the number of $k$-faces a $d$-dimensional polytope with $n$ vertices can have. A couple of years later, Barnette (see \cite{MR298553}, \cite{MR360364}, and \cite{MR328773}) proved the Lower Bound Theorem for {\em simplicial} polytopes; his result provides tight lower bounds on the number of $k$-faces a $d$-dimensional simplicial polytope with $n$ vertices can have. Furthermore, in 1980, Billera and Lee \cite{MR551759} and Stanley \cite{MR563925} completely characterized the face numbers of all simplicial (and by duality also simple) polytopes. Their result is known as the $g$-theorem.  Billera and Lee \cite{MR551759} established sufficiency of the conditions while Stanley \cite{MR563925} proved their necessity.  ​
​

Despite these spectacular advances, to date no Lower Bound Theorem is known for general $d$-dimensional polytopes with an arbitrary number of vertices; in fact, there is not even a plausible conjecture. However for general $d$-dimensional polytopes with $d+s\leq 2d$ vertices, Gr\"unbaum conjectured in \cite[p.~184]{Gr1-2} (see also \cite[p.~265]{MR250188}) that the number of $k$-faces is at least
\[\phi_k(d+s,d) = {d+1 \choose k+1 }+{d \choose k+1 }-{d+1-s \choose k+1 }. \] He proved this conjecture for the cases of $s=2,\;3$, and $4$. The conjecture remained completely open for $s\geq 5$  until very recently Pineda-Villavicencio, Ugon and Yost \cite{MR3899083} proved this conjecture for the number of edges, i.e., they verified the $k=1$ case.​
​

In this paper we prove the conjecture in full generality. Our results can be summarized as follows; 

\smallskip\noindent{\bf Theorem \ref{thm:inequality}} {\em Let $P$ be a $d$-polytope with $d+s$ vertices where $s\geq 2$ and $d\geq s$. Then $f_k(P)\geq \phi_k(d+s,d)$ for every $k$.}\\

In the following, we let $T^{s-1}$ be an $(s-1)$-dimensional simplex, $T^{s}_1$ be the bipyramid over it, and $T^{d,d-s}_1$ be a $(d-s)$-fold pyramid over $T^{s}_1$. We will define these objects once again and in a slower motion in Section \ref{sect: equality}.

\smallskip\noindent{\bf Theorem \ref{thm:equality}} {\em Let $P$ be a $d$-polytope with $d+s$ vertices where $s\geq 2$ and $d\geq s$. If $f_k(P)=\phi_k(d+s,d)$ for some $1\leq k\leq d-2$, then $P$ is $(T^{d,d-s}_1)^*$ --- the polytope that is dual to $T^{d,d-s}_1$.}\\

 The main novelty of our approach is that instead of focusing on contributions coming from facets, we look at sets of potentially unrelated vertices and bound the number of $k$-faces containing one or more of them.

\section{Background and Preliminaries}\label{sect:prelim}
Before starting the proof we recall some definitions and introduce some notation. We refer the reader to books by Gr\"unbaum \cite{Gr1-2} and Ziegler \cite{MR1311028} for all undefined notions. By a polytope we mean the convex hull of finitely many points in $\mathbb{R}^d$. A {\bf simplex} is the convex hull of affinely independent points. A {\bf face} of a polytope $P$ is the intersection of $P$ with a supporting hyperplane. It is known that a face of a polytope is a polytope. The dimension of a polytope is the dimension of its affine span. For brevity, we refer to a $k$-dimensional face as a $k$-face and to a $d$-dimensional polytope as a $d$-polytope. The $0$-faces are called vertices. The $(d-1)$-faces of a $d$-polytope are called facets. We denote by $f_k(P)$ the number of $k$-faces of a polytope $P$.

Let $P\subset \mathbb{R}^d$ be a $d$-polytope and $v$ a vertex of $P$. The {\bf vertex figure of $P$ at $v$}, $P/v$, is obtained by intersecting $P$ with a hyperplane $H$ that separates $v$ from the rest of the vertices of $P$. One property of vertex figures that will be very useful for us is that $(k-1)$-faces of $P/v$ are in bijection with $k$-faces of $P$ that contain $v$.

\section{The proof of the inequality}

We start with the following formulas, most of which are straightforward consequences of Pascal's relation: ${n \choose m} = {n-1 \choose m-1} + {n-1 \choose m}$. For all integers $k$, $d$, and $a>b$, 
\begin{align}\label{eqn: phi formula 1}\phi_{k}(d+a,d) - \phi_{k}(d+b,d) = {d+1-b \choose k+1} - {d+1-a \choose k+1} = \sum_{i=1}^{a-b}{d+1-b-i \choose k};
\end{align}
\iffalse\begin{align}\label{eqn: phi formula 2}\phi_{k}(d+s,d-1) + \phi_{k-1}(d+s,d-1) = \phi_{k}(d+s+1,d);\end{align}\fi
\begin{align}\label{eqn: phi formula 3}\phi_{k}(d+1,d)= {d+1 \choose k+1};\end{align}
\begin{align}\label{eqn: phi formula 4}
\phi_{k}(d+s,d-1) &+ \phi_{k-1}(d-1,d-2)+\phi_{k-1}(d,d-1) \\
& = \phi_{k}(d+s,d-1) + {d-1 \choose k}+{d \choose k}= \phi_{k}(d+s+2,d).\nonumber \end{align}

Let $\mathscr{P}(d+s,d)$ be the set of all $d$-polytopes with $d+s$ vertices. The main ingredient of the proof is the following.
\begin{proposition}\label{thm: key prop}
	Let $P$ be a $d$-polytope and let $\{v_1,v_2,\dots ,v_m\}$ be a subset of vertices of $P$, where $m\leq d$. Then the number of $k$-faces of $P$ that contain at least one of the $v_i$'s is bounded from below by $\sum_{i=1}^{m} {d-i+1 \choose k}$.
\end{proposition}

	\begin{proof}
	We induct on $m$ to show that there exists a sequence of faces, $\{F_1,\;\dots , \; F_m \}$, such that \begin{enumerate}
		\item[(1).] each $F_i$ has dimension $d-i+1$, 
		\item[(2).] $F_i$ contains $v_i$ but does not contain any $v_j$ with $j<i$.
		\end{enumerate}
	
	The base case is $m=1$, and we simply pick $F_1 =P$. Inductively we assume that for every $p\leq m-1$ and any $p$-set of vertices $\{v_1, \;\dots,\; v_p\}$, there exists a sequence $\{F_1, \; \dots ,\; F_p\}$ such that for $1\leq i\leq p$, conditions (1) and (2) are satisfied.
	
	Let $m>1$ and let $v_1,\; \dots , \; v_m$ be $m$ given vertices of $P$. By the inductive hypothesis, for $\{v_1,\;\dots ,\;v_{m-1} \}$ there exist faces $F_1,\;\dots ,\; F_{m-1}$ satisfying conditions (1) and (2). Similarly, by considering $\{v_1,\;\dots ,\;v_{m-2}, v_m \}$, there also exists a $(d-m+2)$-face $F$ that contains $v_m$ but not $v_1,\;\dots,\; v_{m-2}$. Regardless of whether $v_{m-1}$ is in $F$ or not, there must exist a facet of $F$, call it $F_m$,  that contains $v_m$ but not $v_{m-1}$. Then $v_i\in F_m$ if and only if $i = m$, and $F_1,\;\ldots,\; F_{m-1}, F_m$ is a desired sequence.
	
	For each $i$, the $k$-faces of $F_i$ that contain $v_i$ correspond to the $(k-1)$-faces of the vertex figure $F_i/v_i$. Since $\dim(F_i/v_i)=\dim(F_i)-1 =d-i$, we obtain that
	\begin{align*}\#\; k \text{-faces of } P \text{ that contain some } v_i \; (1\leq i\leq m) &\geq 	 \# \;\bigcupdot_{i=1}^{m} \{ k\text{-faces of } F_i \text{ containing } v_i \}\\
	&\geq \#\; 	\bigcupdot_{i=1}^{m} \{ (k-1)\text{-faces } \text{ of } F_i/v_i\}\\
	&\geq  \sum_{i=1}^{m}{d-i+1 \choose k}. 
	\end{align*} 
	The result follows.
	\end{proof}
We are now ready to prove our first main result.
	\begin{theorem}\label{thm:inequality}
	Let $s\geq 3$ and $d\geq s$. Then for all $d$-polytopes $P$ with $d+s$ vertices and for all $1\leq k\leq d-1$, $f_k(P)\geq \phi_k(d+s,d)$.
	\end{theorem}
The statement clearly holds for $s=1$, and the cases of $s=2,\; 3,\;4$ were proved by Gr\"{u}nbaum (see \cite[10.2.2]{Gr1-2}), 
\begin{proof}
The proof is by induction on $s$. We fix $s\geq 2$. The following argument will show that if the statement holds for all pairs $(s', d')$ such that $s'<s$ and $d'\geq s'$, then for all $d\geq s$, it also holds for the pair $(s,d)$. Thus, consider $d\geq s$, and let $P\in \mathscr{P}(d+s,d)$.

If there exists a facet $Q$ of $P$ with $d\leq f_0(Q)\leq d+s-2$, or equivalently if $Q\in \mathscr{P}(d+s-m,d-1)$ where $2\leq m\leq s$, then there are $m$ vertices of $P$ outside of $Q$. We denote them by $\{v_1, v_2, \dots , v_m \}$. The $k$-faces of $P$ fall into two disjoint categories: the $k$-faces of $Q$ and the $k$-faces of $P$ that contain some $v_i$. By the inductive hypothesis, $f_k(Q) \geq \phi_k(d+s-m,d-1)$. Therefore by Proposition \ref{thm: key prop},
 \begin{eqnarray}\label{eq: d+s proof}
 f_k(P) &\overset{(\diamond)}{\geq}& \phi_k(d+s-m,d-1) + \sum_{i=1}^{m} {d-i+1 \choose k}\nonumber \\
 &\overset{\mbox{\tiny (by (\ref{eqn: phi formula 4}))}}{=}& \phi_k(d+s-m+2,d) + \sum_{i=3}^{m}  {d-i+1 \choose k} \nonumber \\
 &\overset{\mbox{\tiny (by (\ref{eqn: phi formula 1}))}}{=}& \left[\phi_k(d+s,d) - \sum_{j=1}^{m-2} {d-s+m-1-j \choose k} \right] + \sum_{i=3}^{m} {d-i+1 \choose k}\\
 &=& \phi_k(d+s,d) + \sum_{j=1}^{m-2} \underbrace{\left[ -{d-j-1-(s-m) \choose k}+ {d-j-1 \choose k}\right]}_{\geq 0 \text{ (since } s\geq m)} \nonumber \\ 
 	&\overset{(\diamond \diamond)}{\geq}& \phi_k(d+s,d).\nonumber
 \end{eqnarray}
This completes the proof of this case. The inequalities $(\diamond)$ and $(\diamond\diamond)$ will be discussed later in the proof of Theorem \ref{thm:equality}.

Otherwise, all facets in $P$ have $d+s-1$ vertices. But this implies that $P$ is a pyramid over each of its facets, and so $P$ can only be a simplex, contradicting our assumption that $P$ has $d+s$ vertices and $s\geq 2$.
%(Indeed, pick any facet $Q\in \mathscr{P}(d+s-1,d-1)$, then $P = v* Q$ for the only vertex $v$ in $P$ that is not in $Q$. Every other facet in $P$ is then $v * F$ for some $(d-2)$-face of $Q$. This implies that every $(d-2)$-face in $Q$ has to be in $\mathscr{P}(d+s-2,d-2)$, but then $Q$ is a pyramid over every one of them. As the argument goes on, eventually we will need all $1-$faces in $P$ to be in $\mathscr{P}(d+s-d+1,1) = \mathscr{P}(s+1,1)$ As $s\geq 2$ this is impossible.)
\end{proof}

\section{Treatment of equality}\label{sect: equality}

In this section we discuss the cases of equality in the Lower Bound Theorem. We first review some definitions relevant to the proof below. For more details, see for example \cite[Chapter 1]{MR1311028}. Let $P\subset \R^{d+1}$ be a $d$-polytope, and let $x\in \R^{d+1}$ be a point that does not lie in the affine hull of $P$. The {\bf pyramid over $P$ with apex $x$} is the convex hull of $P\cup\{x\}$. A pyramid over a $d$-polytope $P$ is a $(d+1)$-polytope. An {\bf $s$-fold pyramid over $P$} is a pyramid over an $(s-1)$-fold pyramid over $P$. (The $0$-fold pyramid of any polytope is the polytope itself.) Similarly, a {\bf bipyramid over $P$} is the convex hull of the union of $P$ and two new vertices $x^+$ and $x^-$ chosen so that they are not in the affine hull of $P$, but the interior of the line segment $[x^+,x^-]$ has a non-empty intersection with the interior of $P$. 

Let $P\subset \R^d$ be a $d$-polytope and let $F$ be a facet of $P$. A point $v\in\R^d$ is {\bf beyond the facet $F$} if the supporting hyperplane of $F$ separates $v$ from $P$.

For every $d$-polytope $P$, there exists a polytope of the same dimension, denoted by $P^*$, whose face lattice is the face lattice of $P$ ``flipped-over" (meaning that the order is reversed). In particular, vertices of $P^*$ correspond to facets of $P$. The polytope $P^*$ (or more precisely, its combinatorial type) is called the {\bf dual polytope of} $P$. 

A vertex of a $d$-polytope is {\bf simple} if it is contained in exactly $d$ facets (equivalently, if it is adjacent to exactly $d$ vertices). A polytope $P$ is simple if all vertices of $P$ are simple. The dual polytope of a simple polytope is a simplicial polytope, and vice versa. 

As part of his proof of the Lower Bound Theorem for simplicial polytopes, Barnette (see \cite{MR298553} and \cite{MR360364}) proved that any simplicial polytope $P$ with at least $d+2$ vertices has at least $2d$ facets. We will use this result in our treatment of equality.

To make the exposition cleaner, we use the following notation of Gr\"unbaum \cite[\S 6.1]{Gr1-2}. \begin{itemize}
\item[] $T^s$ is an $s$-simplex.
\item[] $T^s_m$ is a simplicial $s$-polytope obtained from $T^s$ by adding one additional vertex and placing it beyond exactly $m$ facets of $T^s$, where $1\leq m\leq s-1$. In particular, $T^d_m$  has $d+2$ vertices\footnote{Equivalently, we can define $T^s_m$ as the {\it direct sum} of two simplices: $T^s_m =T^m \oplus T^{s-m} $}.
\item[] $T^{d,d-s}_m$ is a $(d-s)$-fold pyramid over $T^s_m$.
\end{itemize}
Thus, $T^{d,0}_m$ is a $0$-fold pyramid over $T^d_m$, which is $T^d_m$ itself.

Gr\"unbaum \cite[Section 6.1]{Gr1-2} proved the following results, which will be used in the proof of the main result of this seciton.
\begin{lemma}\label{rmk: T^d_m vs T^d_d-m} $T^d_m = T^d_{d-m}$.
\end{lemma}

\begin{lemma}\label{rmk: T^d_m}
	If $P$ is a simplicial $d$-polytope with $d+2$ vertices, then $P = T^d_m$ for some $1\leq m\leq d-1$. 
\end{lemma}

\begin{lemma}\label{rmk: f_d-1 T^d_m}
	For all $0\leq k\leq d-1$, 
	\[f_{k}(T^{d,d-a}_m) = {d+2 \choose d-k+1} -{d-a+m+1 \choose d-k+1}-{d-m+1 \choose d-k+1} +{d-a+1 \choose d-k+1}. \]
	In particular, $f_{d-1}(T^d_m)= f_{d-1}(T^{d,\;0}_m) = d+1 + m(d-m)$. 
\end{lemma}

Now we are ready to state the main result of this section.
\begin{theorem}\label{thm:equality}
  Let $P\in \mathscr{P}(d+s,d)$ where $s\geq 2$ and $d\geq s$. If $f_k(P)=\phi_k(d+s,d)$ for some $k$ with $1\leq k\leq d-2$, then $P= (T^{d,d-s}_1)^*$.
\end{theorem}

		First it is easy to verify that 
		\[ f_k\left((T^{d,d-s}_1)^*\right) = \phi_k(d+s,d) \quad \text{ for all } d\geq s,\; 1\leq k\leq d-1. \]
		
		Assuming $f_k(P)=\phi_k(d+s,d)$ for some $k$ with $1\leq k\leq d-2$, we will prove this theorem using the following corollary of Theorem \ref{thm:inequality}. 
		
\begin{corollary}[Corollary of Theorem \ref{thm:inequality}]\label{cor:diamonds}
	If $f_k(P)=\phi_k(d+s,d)$ for some $k$ with $1\leq k\leq d-2$, then each facet of $P$ has $d$, $d+s-2$, or $d+s-1$ vertices, and $P$ has $d+2$ facets.
\end{corollary}
		
\begin{proof}	
		Notice that (\ref{eq: d+s proof}) holds independently of the choice of a facet (with at most $d+s-2$ vertices) in $P$ or the ordering of the vertices that lie outside of this facet. Thus for $f_k(P)=\phi_k(d+s,d)$ to hold, both inequalities in (\ref{eq: d+s proof}) must be satisfied as equalities for any chosen facet with at most $d+s-2$ vertices. The inequality ($\diamond \diamond$) of (\ref{eq: d+s proof}) holds as equality for some $k < d-1$ if and only if $m=2$ or $s$. This implies that for the equality to hold, each facet of $P$ can only have $d$, $d+s-2$, or $d+s-1$ vertices. The first inequality $(\diamond)$ in (\ref{eq: d+s proof}) holds as equality only if, for every facet in $P$ that has at most $d+s-2$ vertices and for each of the remaining vertices $v_1,\;v_2,\;\dots $,
		\begin{align}\label{eqn: diamond implies}
		\# \{k\text{-faces containing } v_i \text{ but not any } v_j \text{ for } j<i\text{ in } P\} =  {d-i+1 \choose k}.  \end{align}
		Particularly, the number of $k$-faces containing $v_1$ is ${d \choose k}$, hence the number of edges containing $v_1$ is ${d \choose 1} =d$, so $v_1$ is simple. Since $v_1$ is arbitrary, this means that {\bf all} of the vertices that are not in the chosen facet are simple. For each vertex $v$ of $P$ that is not an apex, there is a facet (of size $\leq d+s-2$) that does not contain $v$, so we conclude that every non-apex vertex of $P$ is simple.
		
		We saw in the proof of Theorem \ref{thm:inequality} that it is impossible for all facets of $P$ to contain $d+s-1$ vertices. This means that there must exist facets with $d+s-p$ vertices where $2\leq p \leq s$. Pick such a facet $F$, and label the vertices outside of $F$ as $v_1,\;v_2,\; \dots v_p$. We will show that $f_{d-1}(P)=d+2$. The facets of $P$ fall into the following disjoint categories:\begin{enumerate}
			\item[(0).] $F$;
			\item[(1).] facets containing $v_1$;
			\item[(2).] facets containing $v_2$, but not $v_1$;
			\item[(3).] facets containing $v_3$, but not $v_1,\;v_2$;
			\item[] $\dots $
			\item[(p).] facets containing $v_p$, but not $v_1,\;\dots ,\;v_{p-1}$.
		\end{enumerate}
		
		Since $v_1$ is simple, it is contained in $d$ facets. These facets together with $F$ account for $d+1$ facets of $P$. Next we show that there is a unique facet in category (2), i.e., a unique facet that contains $v_2$, but not $v_1$. Suppose not, and let $F_2$ and $F'_2$ be two distinct facets that contain $v_2$ but not $v_1$. Then there must be a $k$-face of $F'_2$ that contains $v_2$ and is not a face of $F_2$. Therefore
		\begin{align} \# \{k\text{-faces containing } v_2\text{ but not } v_1 \text{ in }P \} \quad &> \quad  \# \{k\text{-faces containing } v_2\text{ but not } v_1 \text{ in } F_2 \} \nonumber \\
		&\geq \quad {d-1 \choose k},\nonumber
		\end{align}
	which contradicts our assumption in (\ref{eqn: diamond implies}).
		
		We have shown that the number of facets of $P$ in categories (0), (1), and (2) is $1+d+1 =d+2$. If $F$ has $d+s-2$ vertices (and so $p=2$), we are done. In the case that $p>2$, it suffices to show that for all $v_i$ with $i\geq 3$, there exist no facets that contain $v_i$ but not $v_1$ and $v_2$. Moreover, by reordering the vertices, it suffices to prove this statement for $v_3$.
		
		Suppose there exists a facet $G$ that contains $v_3$, but not $v_1,\;v_2$. Then
		\begin{align} \# \{k\text{-faces containing } v_3\text{ but not } v_1,\; v_2 \text{ in }P \} \quad &\geq  \quad  \# \{k\text{-faces containing } v_3 \text{ in } G \} \nonumber \\
		&=  \quad  \# \{(k-1)\text{-faces} \text{ of } G/v_3 \} \nonumber \\
		&\geq \quad {d-1 \choose k}.\nonumber\\
		&> \quad {d-2 \choose k} \quad\quad (\text{since }1\leq k\leq d-2).\nonumber
		\end{align}
		This again contradicts our assumption in (\ref{eqn: diamond implies}).
\end{proof}

Now we are ready to prove Theorem \ref{thm:equality}.
\begin{proof}[Proof of Theorem \ref{thm:equality}]	
	Suppose that $P$ has $d-a$ facets that have $d+s-1$ vertices. By Corollary \ref{cor:diamonds}, $P$ has $d+2$ facets. Then $P$ is a $(d-a)$-fold pyramid over an $a$-polytope $Q$ with $a+2$ facets. Since the vertex set of $Q$ consists of non-apex vertices of $P$, by the argument above, $Q$ is a simple polytope. Therefore $P^*$ is a $(d-a)$-fold pyramid over a simplicial $a$-polytope $Q^*$ with $a+2$ vertices. According to Lemma \ref{rmk: T^d_m}, $Q^* = T^{a}_m$ for some $1\leq m\leq a-1$. That is, $Q^*$ is the convex hull of an $a$-simplex $T^a$ and another vertex $v$ beyond $m$ facets of $T^a$. We will show that $a=s$ and $Q^* = T^s_1$.
	
	Since each facet of $P$ has $d+s-1$, $d+s-2$, or $d$ vertices, each facet of $Q$ has either $a+s-2$ or $a$ vertices. Equivalently, each vertex of $Q^* (= T^a_m)$ is contained in $a+s-2$ or $a$ facets. Since $Q$ has $a+s$ vertices, $Q^*$ has $a+s$ facets. Notice that the facets of $T^a_m$ that do not contain $v$ are exactly the facets of $T^a$ which $v$ is not beyond, so the number of such facets is $(a+1)-m$. This, together with the fact that $Q^*$ has $a+s$ facets, implies that
	\begin{align*} \# \; \{\text{facets of } Q^* \text{ that contain } v \} &= f_{d-1}(Q^*)  - \#\; \{\text{facets of } Q^* \text{ that do not contain } v \}\\
	&= \left[a+s\right]  - \left[ (a+1) -m \right]\\
	&= s+m-1.
	\end{align*}
	Since the vertex $v$ is contained in either $a+s-2$ or $a$ facets of $Q^* =T^a_m$, either $s+m-1 =a+s-2$ or $s+m-1= a$.

	If $s+m-1 =a+s-2$, then $m=a-1$, and so $a=s$. 
	
	If $s+m-1 =a$, then since $m\geq 1$ it must be that $s\leq a$. On the other hand, recall that $Q^*$ is a simplicial $a$-polytope with $a+s$ facets (which is not a simplex). Thus, by the Lower Bound Theorem, the number of facets of $Q^*$ is at least twice the dimension of $Q^*$. This means that $a+s \geq  2a$, and so $s\geq a$. Putting this together, we see that the only possibility is that $s=a$. 
	
	We have shown that $a=s$ is the only possible case, therefore $Q^* = T^s_m$ has $a+s =2s$ facets, and so by Lemma \ref{rmk: f_d-1 T^d_m}, we obtain that 
	\[2s = f_{s-1}(Q^*)  = s+1 + m(s-m). \] 
	This equality implies $m=1$ or $s-1$. By Lemma \ref{rmk: T^d_m vs T^d_d-m} $T^s_1 = T^s_{s-1}$. Therefore we conclude that $Q^* = T^s_1$ and so $P= (T^{d,d-s}_1)^*$ as desided.
\end{proof}
An alternative way to end the proof is by using the equality case of the simplicial Lower Bound Theorem (see \cite{MR360364}), which says that $T^s_1$ is the unique simplicial $s$-polytope with $2s$ facets.

\begin{remark}
	Our proof shows that for $d+2 \leq s\leq d$ and $1\leq k\leq d-2$, $(T^{d,d-s}_1)^*$ is the unique polytope in $\mathscr{P}(d+s,d)$ that has $\phi_k(d+s,d)$ many $k$-faces. This is in general not true in the case of $k=d-1$ (where $\phi_{d-1}(d+s,d)=d+2$), i.e., there might exist more than one polytope $(T^{d,d-a}_m)^*\in \mathscr{P}(d+s,d)$ with $d+2$ facets. Among those polytopes, by the theorem above, $(T^{d,d-s}_1)^*$ has the componentwise minimal $f$-vector. This result was also proved in \cite[Theorem 24]{MR3899083}.
\end{remark}

\section*{Acknowledgments}
The author would like to thank Isabella Novik for encouraging her to explore this problem and having taken many hours to discuss and to help revise the first few drafts of this paper. The author is also grateful to Steve Klee for numerous comments and suggestions on the draft, and to G\"unter Ziegler and Guillermo Pineda-Villavicencio for taking the time to look at the previous version of the paper and to provide feedback.

\bibliographystyle{alpha}
\bibliography{general_lbt}

\end{document}